%% file: apaper.tex
\documentclass[11pt]{article}
\usepackage{lmodern, lipsum}
\usepackage{authblk}

\usepackage{fullpage}

\usepackage{amssymb}
\usepackage{amsmath}
\usepackage{bbm}
\usepackage{setspace}
\usepackage{xspace}
\usepackage{enumerate}
\usepackage[square,sort,comma,numbers]{natbib}
\usepackage[nottoc]{tocbibind}
\usepackage{mathrsfs}
\usepackage{csquotes}
\usepackage{tikz-cd}
\usepackage{mathtools}
\usepackage{amsthm}
\usepackage{thmtools}
\usepackage{thm-restate}
\usepackage[hidelinks]{hyperref}
\usepackage{cleveref}

\tikzset{node distance=2cm, auto}

\emergencystretch=\maxdimen
\newtheorem{theorem}{Theorem}[section]
\newtheorem{cor}[theorem]{Corollary}
\newtheorem{lemma}[theorem]{Lemma}
\newtheorem{defn}[theorem]{Definition}
\newtheorem*{defn*}{Definition}
\newtheorem{prop}[theorem]{Proposition}
\newtheorem{question}[theorem]{Question}

\newtheorem{remark}[theorem]{Remark}
\newtheorem{claim}[theorem]{Claim}
\newtheorem{nota}[theorem]{Notation}

\title{Taking Reinhardt's Power Away}
\date{}
\author{Richard Matthews\footnote{email: R.M.A.Matthews@leeds.ac.uk}}
\affil{University of Leeds}
 
\newcommand{\critj}{\textnormal{crit(\textit{j})} }
\newcommand{\critp}{\textnormal{crit\textsubscript{\textit{p}}} }
\newcommand{\critw}{\textnormal{crit\textsubscript{\textit{w}}} }
\newcommand{\crits}{\textnormal{crit\textsubscript{\textit{s}}} }
\newcommand{\cminus}{\raisebox{.33\height}{\scalebox{0.75}{$-$}}}
\newcommand{\uminus}{\raisebox{.1\height}{\scalebox{0.75}{$-$}}}

\newcommand{\Vbold}{\textnormal{\textsc{V}}\xspace} 
\newcommand{\ZFbold}{\textnormal{\textsc{ZF}}\xspace}
\newcommand{\ZFboldminus}{\textnormal{\textsc{ZF}}\cminus\xspace}
\newcommand{\ZFCboldminus}{\textnormal{\textsc{ZFC}}\cminus\xspace}
\newcommand{\ZFCbold}{\textnormal{\textsc{ZFC}}\xspace}
\newcommand{\ZFKbold}{\textnormal{\textsc{ZF(K)}}\xspace}
\newcommand{\ZFCDCminusbold}{\textnormal{\textsc{ZFC}$^{\uminus}$ + \textsc{DC}\textsubscript{$<$\textsc{Ord}}} \xspace}
\newcommand{\Mbold}{\textnormal{\textsc{M}}\xspace}
\newcommand{\Nbold}{\textnormal{\textsc{N}}\xspace}
\newcommand{\Hbold}{\textnormal{\textsc{H}}\xspace}
\newcommand{\Ibold}{\textnormal{\textsc{I}}\xspace}
\newcommand{\Tbold}{\textnormal{\textsc{T}}\xspace}
\newcommand{\Kbold}{\textnormal{\textsc{K}}\xspace}
\newcommand{\Lbold}{\textnormal{\textsc{L}}\xspace}
\newcommand{\Wbold}{\textnormal{\textsc{W}}\xspace}
\newcommand{\DCbold}{\textnormal{\textsc{DC}}\xspace}
\newcommand{\ACbold}{\textnormal{\textsc{AC}}\xspace}

\newcommand{\GBbold}{\textnormal{\textsc{GB}}\xspace}
\newcommand{\HSbold}{\textnormal{\textsc{HS}}\xspace}
\newcommand{\HRbold}{\textnormal{\textsc{HR}}\xspace}
\newcommand{\GBboldminus}{\textnormal{\textsc{GB}}\cminus\xspace}
\newcommand{\Ordbold}{\textnormal{\textsc{Ord}}\xspace}

\DeclareMathOperator{\sym}{\textnormal{\textrm{sym}}}
\DeclareMathOperator{\fix}{\textnormal{\textrm{fix}}}
\DeclareMathOperator{\trcl}{\textnormal{\textrm{trcl}}}
\DeclareMathOperator{\supremum}{\textnormal{\textrm{sup}}}

\DeclareMathOperator{\rank}{\textnormal{\textrm{rank}}}
\DeclareMathOperator{\dom}{\textnormal{\textrm{dom}}}
\DeclareMathOperator{\ran}{\textnormal{\textrm{ran}}}
\DeclareMathOperator{\fld}{\textnormal{\textrm{fld}}}
\DeclareMathOperator{\maximum}{\textnormal{\textrm{max}}}

\DeclareMathOperator{\coll}{\textnormal{\textrm{coll}}}
\DeclareMathOperator{\Add}{\textnormal{\textrm{Add}}}
\DeclareMathOperator{\Col}{\textnormal{\textrm{Col}}}
\DeclareMathOperator{\restrict}{\upharpoonright}
 
\DeclareFontFamily{U}{mathx}{\hyphenchar\font45}
\DeclareFontShape{U}{mathx}{m}{n}{
      <5> <6> <7> <8> <9> <10>
      <10.95> <12> <14.4> <17.28> <20.74> <24.88>
      mathx10
      }{}
\DeclareSymbolFont{mathx}{U}{mathx}{m}{n}
\DeclareFontSubstitution{U}{mathx}{m}{n}
\DeclareMathAccent{\widecheck}{0}{mathx}{"71}

\begin{document}

\maketitle

\input{intro}

\input{noembedding}

\input{bigwithdc}

\input{collectioninsymmod3}

\bibliographystyle{alpha}
\bibliography{paperbib}

\end{document}

%% file: intro.tex


\begin{abstract}
\noindent We study the notion of non-trivial elementary embeddings $j : \Vbold \rightarrow \Vbold$ under the assumption that \Vbold satisfies \ZFCbold without Power Set but with the Collection Scheme. We show that no such embedding can exist under the additional assumption that it is cofinal and either $\Vbold_{\critj}$ is a set or that the Dependent Choice Schemes holds. We then study failures of instances of collection in symmetric submodels of class forcings.
\end{abstract}

\section{Introduction}

A vital tool in many set-theoretic arguments is the assumption that various large cardinal notions are consistent. That is, it is possible to have a cardinal which exhibits certain additional properties that cannot provably exist arguing from \ZFCbold alone. Often these properties can be expressed using the first ordinal moved, or the \emph{critical point}, of an elementary embedding 

\vspace{-5pt}

$$j : \Vbold \rightarrow \Mbold$$

\noindent where \Mbold is some transitive class, with the general principle being that the closer \Mbold is to $\Vbold,$ the stronger the resulting large cardinal assumption. For example, the critical point is said to be $\lambda$-strong if $\Vbold_\lambda \subseteq \Mbold$ or $\lambda$-supercompact if \Mbold is closed under arbitrary sequences of length $\lambda$. \\ \indent There is a natural limit to these large cardinals, originally proposed by Reinhardt \cite{srk}, which is for $\kappa$ to be the critical point of a non-trivial elementary embedding from \Vbold to itself. However, as shown by Kunen \cite{kun71} there is no such non-trivial embedding when \Vbold is a model of \ZFCbold along with a predicate for $j$ such that \Vbold satisfies all instances of replacement and separation in the language expanded to include this predicate. In fact, Kunen's proof shows that there is no non-trivial elementary embedding

\vspace{-5pt}

$$j : \Vbold_{\lambda + 2} \rightarrow \Vbold_{\lambda + 2}$$

\noindent for any ordinal $\lambda$ under the assumption that \Vbold satisfies $\ZFCbold.$ \\ 

\noindent Since the announcement of this result it has been a long standing and much studied question as to whether or not the axiom of choice is necessary for this result. Namely, if it is consistent for there to be a non-trivial elementary embedding $j : \Vbold \rightarrow \Vbold$ under the assumption that \Vbold is a model of $\ZFbold.$ In this paper we take a different approach to generalising Kunen's inconsistency which is to study such embeddings in the theory \ZFCbold \emph{without Power Set}. The motivation for this is the following result which shows that such embeddings are consistent under the assumption that $\Ibold_1$ is consistent, which is an axiom just short of the Kunen inconsistency. In particular, $\Ibold_1$ will give a non-trivial embedding from $\Hbold_{\lambda^+}$ to itself, which is one of the standard structures which models \ZFCbold without Power Set.

\begin{restatable*}{theorem}{characterisingione}
\label{characterising I1}
There exists an elementary embedding $k : \Vbold_{\lambda + 1} \rightarrow \Vbold_{\lambda + 1}$ if and only if there exists an elementary embedding $j : \Hbold_{\lambda^+} \rightarrow \Hbold_{\lambda^+}$.
\end{restatable*}

\noindent It is a well-known result that without Power Set many of the usual equivalent ways to formulate the axioms of \ZFCbold break down. In particular, the Replacement Scheme no longer implies the stronger Collection Scheme and the Axiom of Choice does not imply that every set can be well-ordered. Without the Collection Scheme many of the basic facts that one assumes no longer hold, for example we can consistently have that $\omega_1$ exists and is singular or that the \L o\'{s} ultrapower theorem can fail. One can find these and other similar results in \cite{ghj}. So, using the notation found in that paper, we shall define \emph{the theory} \ZFCbold \emph{without Power Set} as follows:

\begin{defn} Let $\ZFboldminus$ denote the theory consisting of the following axioms: Empty set, Extensionality, Pairing, Unions, Infinity, the Foundation Scheme, the Separation Scheme and the Replacement Scheme. 

$\ZFbold^{\uminus}$ denotes the theory $\ZFboldminus$ plus the Collection Scheme.

$\ZFCbold^{\uminus}$ denotes the theory $\ZFbold^{\uminus}$ plus the Well-Ordering Principle.

$\ZFCbold_{Ref}^{\uminus}$ denotes the theory $\ZFCbold^{\uminus}$ plus the Reflection Principle.

\ZFCDCminusbold denotes the theory $\ZFCbold^{\uminus}$ plus the $\DCbold_\mu$-Scheme for every cardinal $\mu$.
\end{defn}

We will also use the corresponding notation for their second order versions $\GBbold^{\uminus}$ and \GBboldminus. \\

\noindent The main result of this paper is that, under mild assumptions, Kunen's inconsistency does still hold in the theory $\ZFCbold^{\uminus} _j$ (the theory $\ZFCbold^{\uminus}$ with a predicate for $j$) and therefore that, while $\Hbold_{\lambda^+}$ has more structure than $\Vbold_{\lambda + 1}$, the embedding given by $\Ibold_1$ cannot have one of the most useful properties an embedding can have, cofinality.

\begin{restatable*}{theorem}{nocofinal}
\label{no cofinal}
There is no non-trivial, cofinal, $\Sigma_0$-elementary embedding $j : \Vbold \rightarrow \Vbold$ such that $\Vbold \models \ZFCbold^{\uminus}_j$ and $\Vbold_{\critj} \in \Vbold$.
\end{restatable*}

\noindent The proof of the above theorem makes essential use of the fact that the initial segment of the universe up to the critical point of $j$, $\Vbold_{\critj}$ is a set. However, by strengthening the underlying theory to also satisfy the \emph{Dependent Choice Scheme of length} $\mu$ for every cardinal $\mu$, this can be removed.

\begin{restatable*}{cor}{nonewithref}
There is no non-trivial, cofinal, $\Sigma_0$-elementary embedding $j : \Vbold \rightarrow \Vbold$ such that $\Vbold \models (\ZFCDCminusbold)_j$.
\end{restatable*}

\noindent The second half of the paper deals with a curious property of the theory \ZFCDCminusbold which is needed for the above result. This is the property that if \Vbold is a model of \ZFCDCminusbold and $\mathcal{C}$ is a definable proper class over \Vbold then for any non-zero ordinal $\gamma$ there is a definable surjection of $\mathcal{C}$ onto $\gamma$. A reasonable suggestion for a counter-example to this property without choice is to construct a model of $\ZFbold^{\uminus}$ with an amorphous class, where 

\begin{defn}
An infinite class $A$ is said to be \emph{amorphous} if it cannot be partitioned into two infinite classes.
\end{defn}

\noindent Note that such a class could not surject onto $\omega$ otherwise it could be partitioned into the class mapped to even numbers and the class mapped to odd numbers.

However, we show that, without choice, many structures which would otherwise satisfy Collection only satisfy the Replacement Scheme. This is done by proving the following result which shows that the existence of an amorphous proper class implies the failure of collection and thus a failure of our hoped-for counter-example. Moreover, using this we can show that a symmetric submodel of a pretame class forcing need not model the Collection Scheme.

\begin{restatable*}{theorem}{collectionfails}
\label{Collection Fails} Suppose that $\langle \Mbold, A \rangle$ satisfies;
\begin{enumerate}
\item $\Mbold \models (\ZFboldminus)_A$,
\item $A \subseteq \Mbold$ and $\langle \Mbold, A \rangle \models ``A \textit{ is a proper class''}$,
\item $\langle \Mbold, A \rangle \models \textit{`` if } B \subseteq A \textit{ is infinite then B is a proper class''}$.
\end{enumerate}
Then the Collection Scheme fails in $\langle \Mbold, A \rangle$. Moreover, $\langle \Mbold, A \rangle$ does not have a cumulative hierarchy and therefore the Power Set also fails.
\end{restatable*}

\noindent \textbf{Acknowledgements} This paper is part of the author's PhD thesis, supervised by Andrew Brooke-Taylor and Michael Rathjen. I am grateful to Philipp Schlicht for pointing out the use of dependent choice in the original proof of Theorem \ref{classes are big}, Johannes Sch\"{u}rz whose suggestion of considering amorphous classes led to the second half of this paper and Asaf Karagila for many insightful conversations about symmetric submodels. I would also like to thank Bea Adam-Day and John Howe for sitting through and then giving feedback on various early iterations of this work.


%% file: noembedding.tex

\newpage
\section[Embeddings of H lambda+]{Embeddings of $\Hbold_{\lambda^+}$}

\begin{defn} $\Ibold_1$ is the assertion that there exists a non-trivial elementary embedding \mbox{$k : \Vbold_{\lambda + 1} \rightarrow \Vbold_{\lambda + 1}.$}
\end{defn}

\noindent $\Ibold_1$ is considered one of the strongest large cardinal axioms that is not known to be inconsistent. The following result is adapted from the folklore result which gives an alternate characterisation of 1-extendible cardinals, a proof of which can be found in \cite{bt}. This theorem shows an equivalent way of considering $\Ibold_1$ embeddings as embeddings of $\Hbold_{\lambda^+}$, a set with much more structure than $\Vbold_{\lambda + 1}$.

\characterisingione

\begin{proof} $(\Leftarrow):$ By the Kunen inconsistency, $\lambda$ must be the supremum of the critical sequence $\langle \kappa_n : n \in \omega \rangle$ of $j$, where $\kappa_0$ is the critical point and $\kappa_{n + 1} = j(\kappa_n)$. Then each $\kappa_n$ is an inaccessible cardinal and thus $2^{< \lambda} = \lambda = \beth_\lambda$. Therefore $\Vbold_\lambda = \Hbold_\lambda$ and $| \Vbold_\lambda | = \lambda$ so $\Vbold_\lambda \in \Hbold_{\lambda^+}$. This means that $\Vbold_{\lambda + 1} = \{ x \in \Hbold_{\lambda^+} : x \subseteq \Vbold_{\lambda} \}$, so $\Vbold_{\lambda + 1}$ is a definable class in $\Hbold_{\lambda^+}$. Moreover, working in $\Hbold_{\lambda^+}$, any formula $\varphi$ can be relativised to $\Vbold_{\lambda + 1}$ so $j \hspace{-0.3pt} \restrict \Vbold_{\lambda + 1} : \Vbold_{\lambda + 1} \rightarrow \Vbold_{\lambda + 1}$ is elementary. \\

\noindent $(\Rightarrow):$ We begin by defining a standard way to code elements of $\Hbold_{\lambda^+}$ by elements of $\Vbold_{\lambda+ 1}$. This will be done by coding $\trcl(\{x\})$ by some subset of $\lambda \times \lambda$ whose Mostowski collapse is again $\trcl(\{x\})$. However, since we will be working with $\trcl(\{x\})$ rather than $x$ itself, it is necessary to do a simple, preliminary coding.

So, let $\hat{\Hbold} \coloneqq \{ \trcl(\{x\}) : x \in \Hbold_{\lambda^+} \}$. For $\trcl(\{x\}), \trcl(\{y\}) \in \hat{\Hbold}$ define the relation $\hat{\in}$ by $\trcl(\{x\}) ~ \hat{\in} ~ \trcl(\{y\})$ if and only if $x \in y$ and similarly for $\hat{=}$. It is then clear that any first order statement $\varphi$ about $\Hbold_{\lambda^+}$ is equivalent to a formula $\hat{\varphi}$ over $\hat{\Hbold}$ by the obvious coding. \\

\noindent Now note that $\rank(\lambda \times \lambda) = \lambda$ and so any subset of $\lambda \times \lambda$ has rank at most $\lambda$. So, for any $x \in \Hbold_{\lambda^+}$ and bijection $f : |\trcl(\{x\})| \rightarrow \trcl(\{x\})$ let $$ C_{x,f} \coloneqq \{ \langle \alpha, \beta \rangle \in \lambda \times \lambda : f(\alpha) \in f(\beta) \} \in \Vbold_{\lambda + 1}.$$ Then the Mostowski collapse of $C_{x,f}$, $\coll(C_{x,f})$, is $\trcl(\{x\})$. Let $\tilde{H}$ denote the definable class in $\Vbold_{\lambda + 1}$ of all subsets of $\lambda \times \lambda$ which code an element of $\hat{\Hbold}$ in this way. That is $X \in \tilde{H}$ iff $X$ is a well-founded, extensional, binary relation on $\lambda$ with a single maximal element and $\dom(X) \cup \ran(X)$ is a cardinal which is at most $\lambda$. \\

\noindent For $Z \in \tilde{H}$, let $\fld(Z)$ be $\dom(Z) \cup \ran(Z)$ and define $\maximum(Z)$ to be the unique element of $\fld(Z)$ which is maximal with respect to the relation on $Z$. Now, for $X, Y \in \tilde{H}$ define relations $\tilde{=}$ and $\tilde{\in}$ by: 

\vspace{-5pt}

\begin{gather*}
X ~ \tilde{=} ~ Y  \Longleftrightarrow \exists g : \lambda \rightarrow \lambda ~ (g ~ \textit{is a bijection  } \wedge ~ \forall \alpha, \beta \in \lambda \\
\phantom{X ~ \tilde{=} ~ Y  \Longleftrightarrow \exists g : \lambda \rightarrow} (\langle \alpha, \beta \rangle \in X \leftrightarrow \langle g(\alpha), g(\beta) \rangle \in Y)) 
\end{gather*} 

\vspace{-5pt}

\begin{gather*}
X ~ \tilde{\in} ~ Y \Longleftrightarrow \exists g : \lambda \rightarrow \lambda ~ (g \textit{ is injective } \wedge ~ \langle g(\maximum(X)), \maximum(Y) \rangle \in Y \\
\phantom{X ~ \tilde{\in} ~ Y \Longleftrightarrow \exists g :} \wedge ~ \forall \alpha, \beta \in \fld(X) ~ (\langle \alpha, \beta \rangle \in X \leftrightarrow \langle g(\alpha), g(\beta) \rangle \in Y).
\end{gather*}

\vspace{5pt}

Then $~ \tilde{=} ~$ and $~ \tilde{\in} ~$ are definable in $\Vbold_{\lambda + 1}$, with $X ~ \tilde{=} ~ Y \Longleftrightarrow \coll(X) = \coll(Y)$ and $X ~ \tilde{\in} ~ Y \Longleftrightarrow \coll(X) \in \coll(Y)$. Now we have that any first order statement $\hat{\varphi}$ about $\hat{\Hbold}$ is equivalent to a formula $\tilde{\varphi}$ over $\Vbold_{\lambda + 1}$ which is defined by the following coding: \begin{itemize}
\item Replace any parameter $\trcl(\{x\})$ occurring in $\hat{\varphi}$ with $C_{x, f}$ for some (any) bijection $f : |\trcl(\{x\})| \rightarrow \trcl(\{x\})$.
\item Replace any instance of $\hat{=}$ with $\tilde{=}$ and $\hat{\in}$ with $\tilde{\in}$.
\item Replace any unbounded quantification by the same quantifier taken over $\tilde{H}$.
\end{itemize}
Then, by the elementarity of $k$, 
\begin{center}

\vspace{-20pt}
\begin{eqnarray*}
X ~ \tilde{=} ~ Y  & \Longleftrightarrow &  k(X) ~ \tilde{=} ~ k(Y) \\
& \textrm{ and } & \\
X ~ \tilde{\in} ~ Y & \Longleftrightarrow & k(X) ~ \tilde{\in} ~ k(Y).
\end{eqnarray*}
\end{center}

\noindent Also, since $\tilde{\Hbold}$ is a definable class in $\Vbold_{\lambda + 1}$ the restriction of the embedding $k \restrict \tilde{\Hbold} : \tilde{\Hbold} \rightarrow \tilde{\Hbold}$ is still elementary.

 So we can define $j : \Hbold_{\lambda^+} \rightarrow \Hbold_{\lambda^+}$ by setting $j(x)$ to be the unique element of $\coll(k(C_{x,f}))$ of maximal rank for some bijection $f : | \trcl(\{x\})| \rightarrow \trcl(\{x\})$. Moreover $j$ is elementary since

\vspace{-70pt}

\begin{flushright}
{\setstretch{2.5}
\begin{align*} \Hbold_{\lambda^+}  \models \varphi(x_1, \dots, x_n) & \Longleftrightarrow \hat{\Hbold} \models \hat{\varphi}(\trcl(\{x_1\}), \dots, \trcl(\{x_n\})) \\ \phantom{\Hbold_{\lambda^+} \models \varphi(x_1)} & \Longleftrightarrow  \tilde{\Hbold} \models \tilde{\varphi} \big( C_{x_1,f_1}, \dots, C_{x_n, f_n} \big) \\ \phantom{\Hbold_{\lambda^+} \models \varphi(x_1)} & \Longleftrightarrow \tilde{\Hbold} \models \tilde{\varphi} \big( k(C_{x_1,f_1}), \dots, k(C_{x_n, f_n}) \big) \\ \phantom{\Hbold_{\lambda^+} \models \varphi(x_1)} & \Longleftrightarrow \hat{\Hbold} \models \hat{\varphi} \big( \coll(k(C_{x_1,f_1})), \dots, \coll(k(C_{x_n, f_n})) \big) \\ \phantom{\Hbold_{\lambda^+} \models \varphi(x_1)} & \Longleftrightarrow \Hbold_{\lambda^+} \models \varphi(j(x_1), \dots, j(x_n)). \qedhere
\end{align*}}  \end{flushright} \end{proof}

\noindent The above theorem shows that the existence of a non-trivial elementary embedding from $\Vbold$ to itself under $\ZFCbold^{\uminus}$ is weaker than $\Ibold_1$, however it does not show that the embedding one obtains has any useful structure. What we shall show is that this embedding must fail one of the most useful fundamental characteristics, that of \emph{cofinality} where

\begin{defn} An embedding $j : \Mbold \rightarrow N$ is said to be cofinal if for every $y \in \Nbold$ there is an $x \in \Mbold$ such that $y \in j(x)$.
\end{defn}

\begin{remark}
If \Mbold satisfies \ZFbold and $\Nbold \subseteq \Mbold$ then the cumulative hierarchy for \Mbold witnesses that any elementary embedding is cofinal.
\end{remark}

\newpage
\section{Definable Embeddings}

We begin this section with a standard fact about non-trivial elementary embeddings which is that they must move an ordinal. The only notable thing about the statement is that the proof only requires elementarity for bounded formulae. 

\begin{prop} \label{crit exists}
Suppose that $\Mbold \models \ZFbold^{\uminus}$, $\Nbold \subseteq \Mbold$ is a transitive class model of $\ZFbold^{\uminus}$ and \mbox{$j : \Mbold \rightarrow \Nbold$} is a non-trivial, $\Sigma_0$-elementary embedding. Then there exists an ordinal $\alpha$ such that $j(\alpha) > \alpha$.
\end{prop}

\begin{proof}
First note that $\Sigma_0$-elementarity implies $\Delta_1$-elementarity. Now, since being an ordinal is $\Sigma_0$ definable, if $\alpha$ is an ordinal then so is $j(\alpha)$. Next, since $\emptyset$ is definable as the unique set $z$ such that $\forall y \in z ~ (y \neq y)$, which is a $\Sigma_0$ formula, $j(\emptyset) = \emptyset$. So, by induction, we have that for every ordinal $\alpha$, $j(\alpha) \geq \alpha$. Now let $x$ be a set of least rank such that $j(x) \neq x$ and let $\delta = \rank(x)$. Then for all $y \in x$, $y = j(y) \in j(x)$ so $x \subseteq j(x)$. Thus there must be some $z \in j(x) \setminus x$. Now suppose that $\rank(j(x)) = \delta$, then we must have that $j(z) = z \in j(x)$ so, by elementarity, $z \in x$ which yields a contradiction. Hence, since the following is $\Delta_1$ definable, we must have that $j(\delta) = \rank(j(x)) > \delta$.  
\end{proof}

\noindent Therefore, given a non-trivial, $\Sigma_0$-elementary $j : \Mbold \rightarrow \Nbold$, we will define the critical point of $j$ to be the least ordinal moved and denote it by $\critj$. 

It is not a priori obvious that being an elementary embedding should be definable by a single sentence. However, as proven by Gaifman in \cite{gai}, if $\Mbold$ is a model of a sufficient fragment\footnote{Gaifman's original proof is done under the assumption that \Mbold is a model of Zermelo, that is \ZFbold with separation but not replacement. He then comments that the assumption of Power Set can be replaced by the existence of Cartesian products.} of \ZFCbold then it suffices to check that a cofinal embedding is elementary for $\Sigma_0$ sentences. The version below for the case where $\Mbold \models \ZFbold^{\uminus}$ appears in \cite{ghj}.

\begin{theorem}[Gaifman] \label{fully elementary}
Suppose that $\Mbold$ is a model of $\ZFbold^{\uminus}$ and $j : \Mbold \rightarrow \Nbold$ is a cofinal, \mbox{$\Sigma_0$-elementary} embedding. Then $j$ is fully elementary.
\end{theorem}

\begin{remark} This theorem does not require any assumptions on $\Nbold.$ Moreover, the models \Mbold and \Nbold need not be transitive. \end{remark}

\noindent Using the fact that being $\Sigma_0$-elementary is definable by a single formula we obtain a version of Suzuki's theorem on the non-definability of embeddings, \cite{suz}, in the context of $\ZFbold^{\uminus}$. 

\begin{theorem}[Suzuki] \label{no definable}
Assume that $\Vbold \models \ZFbold^{\uminus}$. Then there is no non-trivial, cofinal, elementary embedding $j : \Vbold \rightarrow \Vbold$ which is definable from parameters.
\end{theorem}

\begin{proof}
Formally, this is a theorem scheme asserting that for each formula $\varphi$ there is no parameter $p$ for which $\varphi(\cdot, \cdot, p)$ defines a non-trivial, cofinal, elementary embedding $j : \Vbold \rightarrow \Vbold$. Using Theorem \ref{fully elementary}, it suffices to show that for no parameter $p$ are we able to define a non-trivial, cofinal, \mbox{$\Sigma_0$-elementary} embedding $j : \Vbold \rightarrow \Vbold$ by $$j(x) = y \Longleftrightarrow \varphi(x, y, p) ~ \textrm{holds.}$$ 

So, seeking a contradiction, let $\sigma(p)$ be the sentence asserting that $\varphi(\cdot, \cdot, p)$ defines a $\Sigma_0$-elementary embedding and let $\psi(p)$ asserts that $\varphi(\cdot, \cdot, p)$ defines a total function which is non-trivial, cofinal and $\Sigma_0$-elementary. That is, $$\psi(p) \equiv \forall x \exists ! y ~ \varphi(x, y, p) ~ \wedge \exists x ~ \neg \varphi(x, x, p) ~ \wedge ~ \forall y \exists x, z ~ (\varphi(x, z, p) ~ \wedge ~ y \in z) ~ \wedge ~ \sigma(p)$$ Let $\vartheta(p, \kappa)$ assert that $\kappa$ is the critical point of $j$. So, $$\vartheta(p, \kappa) \equiv \kappa \in \Ordbold ~ \wedge ~ \forall \alpha \in \kappa ~ \varphi(\alpha, \alpha, p) ~ \wedge ~ \neg \varphi(\kappa, \kappa, p).$$  Then, by Proposition \ref{crit exists}, $$\Vbold \models \psi(p) \rightarrow \exists ! \kappa ~ \vartheta(p, \kappa).$$ So denote by \critp the (unique) $\kappa$ for which $\vartheta(p, \kappa)$ holds. Now fix $p$ such that \critp is as small as possible, that is such that $$\Vbold \models \psi(p) ~ \wedge ~ \forall w ~ (\psi(w) \rightarrow \critp \leq \critw).$$ Then, by elementarity, $$\Vbold \models \exists s ~ \varphi(p, s, p) ~ \wedge ~ \psi(s) ~ \wedge ~ \forall w (\psi(w) \rightarrow \crits \leq \critw)$$ But, $\Vbold \models \critp < \crits$ because the critical point of the embedding defined by $\varphi(\cdot, \cdot, s)$ must be $j(\critp)$, yielding a contradiction.
\end{proof}

\noindent We remark here that the main element of the proof was that being fully elementary can be expressed in a single sentence. Therefore by using Gaifman's original theorem, since the Collection Scheme wasn't used in the proofs of \ref{crit exists} and \ref{no definable}, the above proof also shows that there is no non-trivial cofinal elementary embedding of a model of Zermelo into itself which is definable from parameters. There are two obvious questions which appear here about whether or not the assumptions of cofinality and collection were necessary in the proof that there is no definable embedding. That is;

\begin{question}
Are either of the following two statements consistent:
\begin{enumerate}
\item There exists a non-trivial elementary embedding $j : \Vbold \rightarrow \Vbold$ which is definable from parameters where $\Vbold \models \ZFbold^{\uminus}$?
\item There exists a non-trivial, cofinal elementary embedding $j : \Vbold \rightarrow \Vbold$ which is definable from parameters where $\Vbold \models \ZFboldminus$?
\end{enumerate} 
\end{question}

\noindent It has been proven in \cite{ghj} that one can have cofinal, $\Sigma_1$-elementary embeddings of $\ZFboldminus$ which are not $\Sigma_2$-elementary, which is to say that Gaifman's Theorem can fail without the Collection Scheme. Therefore proving Suzuki's Theorem in either of these contexts would involve a different approach.

\newpage
\section{Choosing from Classes}

In this short section we mention how one can apply choice to set-length sequences of classes using the Collection Scheme. The standard way to do this in full \ZFCbold is by using Scott's trick  to replace each class by the set of elements of least rank of that class. However, if $\Vbold_\alpha$ is not a set for each $\alpha$ then this may not be possible so we have to be slightly more careful in our approach. 

Let $\mu$ be an ordinal and suppose that we have a sequence of non-empty classes $\langle \mathcal{C}_\alpha : \alpha \in \mu \rangle$ which are uniformly defined. This allows us to fix a formula $\varphi(v_0, v_1)$ saying that $v_1 \in \mathcal{C}_{v_0}$. Then, for each $\alpha \in \mu$ there is some set $x$ such that $\varphi(\alpha, x)$. So, by collection, there is some set $b$ such that for each $\alpha \in \mu$ there is some $x \in b$ such that $\varphi(\alpha, x)$. By well-ordering $b$, there is some cardinal $\tau$ and bijection $h : \tau \leftrightarrow b$. So for each $\alpha \in \mu$ we can define a choice function by taking $x_\alpha \in \mathcal{C}_\alpha$ to be $h(\gamma)$ for the least ordinal $\gamma \in \tau$ such that $\varphi(\alpha, h(\gamma))$. \\ \indent For example, suppose that $S \subseteq \mu$ were a stationary set which was partitioned into $\tau < \mu$ many sets $\langle S_\alpha : \alpha \in \tau \rangle$ and one wanted to show that for some $\alpha \in \tau$, $S_\alpha$ was stationary. Arguing for a contradiction, suppose that none of the $S_\alpha$ were stationary and for each $\alpha \in \tau$ define $\mathcal{C}_\alpha$ to be the non-empty class of clubs $D \subseteq \mu$ for which $D \cap S_\alpha \neq \emptyset$. By the above argument, we can choose a sequence of clubs $ \langle D_\alpha : \alpha \in \tau \rangle$ such that for each $\alpha$, $D_\alpha \in \mathcal{C}_\alpha$. Then $\bigcap_{\alpha \in \tau} D_\alpha \cap S = \emptyset$ yielding the required contradiction.

Using this idea we are able to prove many useful classical results without much change from their standard proofs. For completeness, we give here two such $\ZFCbold^{\uminus}$ results which we will then use in our proof of the Kunen inconsistency. 

\begin{theorem}[Fodor] Let $\mu$ be a regular cardinal, $S \subseteq \mu$ stationary and $f$ a regressive function on $S$.\footnote{A function $f : S \rightarrow \Ordbold$ is \emph{regressive} if for any non-zero $\alpha \in S$, $f(\alpha)< \alpha$.} Then there exists some stationary set $T \subseteq S$ and $\gamma \in \mu$ such that for all $\alpha \in T$, $f(\alpha) = \gamma$.
\end{theorem}

\begin{proof}
Assume for a contradiction that for each $\gamma \in \mu$ the set $\{ \alpha \in S : f(\alpha) = \gamma \}$ was non-stationary. Using the above comments, for each $\gamma \in \mu$ choose a club $D_\gamma$ such that for each $\alpha$ in $D_\gamma \cap S$, $f(\alpha)\neq \gamma$. Let $$D = \Delta_{\gamma \in \mu} D_\gamma \coloneqq \{ \alpha : \forall \beta \in \alpha ~ (\alpha \in D_\beta) \}$$ and note that this is club in $\mu$. Therefore $S \cap D$ is stationary, so in particular non-empty, and for any $\alpha \in S \cap D$ and $\gamma \in \alpha$, $f(\alpha) \neq \gamma$. So $f(\alpha) \geq \alpha$, contradicting the assumption that $f$ was regressive.
\end{proof}

\begin{defn} For cardinals $\delta < \mu$ let $S^\mu_\delta = \{ \alpha < \mu : cf(\alpha) = \delta \}$. \end{defn}

\begin{theorem}[Solovay]\label{solovay partition} Suppose that $\mu$ is an uncountable, regular cardinal and $S \subseteq S^\mu_\omega$ is stationary. Then there is a partition of $S$ into $\mu$ many disjoint stationary sets. \end{theorem}

\begin{proof}
First note that for each $\alpha \in S$ there is some increasing sequence of ordinals $\langle t_n : n \in \omega \rangle$ cofinal in $\alpha$. Therefore by the comments at the beginning of this section, for each $\alpha \in S$ choose an increasing sequence $\langle a^\alpha_n : n \in \omega \rangle$ cofinal in $\alpha$. Then, as in the usual proof, using our first example and the regularity of $\mu$ we can fix some $n \in \omega$ such that for each $\sigma \in \mu$, $\{ \alpha \in S : a^\alpha_n \geq \sigma \}$ is stationary in $\mu$. Now define a regressive function $f : S \rightarrow \mu$ by $f(\alpha) = a^\alpha_n$. Using Fodor's Theorem, for each $\sigma \in \mu$ fix some $S_\sigma$ stationary and $\gamma_\sigma \geq \sigma$ such that for all $\alpha \in S_\sigma$, $f(\alpha) = \gamma_\sigma$. Then if $\gamma_\sigma \neq \gamma_{\sigma'}$, $S_\sigma \cap S_{\sigma'} = \emptyset$ and, by the regularity of $\mu$, $| \{ S_\sigma : \sigma \in \mu \} | = \mu,$ which gives the required partition. 
\end{proof}

\section{Non-existence of embeddings}

We are now in the position to prove that there is no non-trivial, cofinal elementary embedding $j$ of $\ZFCbold^{\uminus}$  with $\Vbold_{\critj} \in \Vbold$. This shall be done in two parts; first we shall show that Woodin's proof of the Kunen inconsistency, which is the second proof of Theorem 23.12 in \cite{kan}, goes through in $\ZFCbold^{\uminus}$ under the additional assumption that $(\supremum\{j^n(\critj) : n \in \omega\})^+ \in \Vbold$. Then we shall show, by modifying the coding from \Cref{characterising I1}, that no cofinal embedding can exist in any model that sufficiently resembles $\Hbold_{\lambda^+}$. 

In order to prove this formally we shall work in a subtheory of the second order version of $\ZFCbold^{\uminus}$ which we shall denote as $\ZFCbold^{\uminus}_j$. Essentially this is the theory $\ZFCbold^{\uminus}$ along with a class predicate $j$ for the elementary embedding and the assertion that the Collection and Separation Schemes hold when expanded to include formulae with this predicate. To be more precise;

\begin{defn} Suppose that \Tbold is a first order theory, $\Mbold$ is a model of $\Tbold$ and $j$ is a class predicate. We say that $\Mbold \models \Tbold_j$ if \Mbold satisfies the Replacement / Collection / Reflection and Separation Schemes of \Tbold in the language expanded to include $j$. 
\end{defn}

\begin{theorem}\label{none with lambda+}
There is no non-trivial, elementary embedding $j : \Vbold \rightarrow \Vbold$ such that $\Vbold \models \ZFCbold^{\uminus}_j$ and $(\supremum\{j^n(\critj) : n \in \omega\})^+ \in \Vbold$.
\end{theorem}




\begin{proof}
Suppose for a contradiction that $j : \Vbold \rightarrow \Vbold$ was a non-trivial elementary embedding with critical point $\kappa$ and let $\lambda = \supremum \{ j^n(\kappa) : n \in \omega\}$. Then $j(\lambda) = \lambda$ and, since $\lambda^+$ is definable as the least cardinal above $\lambda$, $j(\lambda^+) = \lambda^+$. Now, using Theorem \ref{solovay partition}, let $\langle S_\alpha: \alpha \in \kappa \rangle$ be a partition of $S^{\lambda^+}_\omega$ into $\kappa$ many disjoint stationary sets and let $S = \{ \langle \alpha, S_\alpha \rangle : \alpha \in \kappa \}$. Then $j(S) = \{ \langle \alpha, T_\alpha \rangle : \alpha \in j(\kappa) \}$ and, by elementarity, $\langle T_\alpha : \alpha \in j(\kappa) \rangle$ is a partition of $S^{\lambda^+}_\omega$ into disjoint sets such that for each $\alpha$ $$ T_\alpha \textit{ is a stationary subset of } \lambda^+. $$ Also, we have that for each $\alpha \in \kappa$, $j(S_\alpha) = T_\alpha$. We claim that there is some $\beta \in \kappa$ such that $T_\kappa \cap S_\beta$ is stationary. For suppose not, then by our comments on choosing from set many classes, for each $\alpha$ we can fix a club $C_\alpha$ such that $T_\kappa \cap S_\alpha \cap C_\alpha = \emptyset$. Letting $C = \bigcap_{\alpha \in \kappa} C_\alpha$ we must have that $$\emptyset = T_\kappa \cap C \cap \bigcup_{\alpha \in \kappa} S_\alpha = T_\kappa \cap C, $$ contradicting the assumption that $T_\kappa$ was stationary. So fix $\beta$ such that $T_\kappa \cap S_\beta$ is stationary. Now, let $$U = \{ \gamma \in \lambda^+ : \gamma = j(\gamma) \}$$ and note that $U$ is closed under sequences of length $\omega$. Therefore there exists some $\sigma \in U \cap T_\kappa \cap S_\beta$. But then $\sigma = j(\sigma) \in j(S_\beta) = T_\beta$, contradicting the assumption that the $T_\alpha$ were disjoint. Hence no such embedding can exist.
\end{proof}

\begin{remark}
The above theorem did not require any assumption about $j$ being cofinal or that $\Vbold_{\critj}$ was a set.
\end{remark}

\nocofinal

\begin{proof}
Suppose for a contradiction that $j : \Vbold \rightarrow \Vbold$ was a non-trivial, cofinal, $\Sigma_0$-elementary embedding with critical point $\kappa$ and let $\lambda = \supremum \{ j^n(\kappa) : n \in \omega\}$. Note that, by an instance of replacement with the parameter $j$, $\lambda \in \Vbold$. Now there are two cases: \begin{itemize}
\item \textbf{Case 1:} $\lambda^+$ exists.
\item \textbf{Case 2:} For all $x \in \Vbold$, there is an injection $f : x \rightarrow \lambda$.
\end{itemize}

\noindent \textbf{Case 1:}
This is just a special case of Theorem \ref{none with lambda+}.

\noindent \textbf{Case 2:}

\vspace{12pt}

First note that by elementarity, since $\Vbold_\kappa \in \Vbold$ so is $\Vbold_{j^n(\kappa)}$ for each $n \in \omega$ and therefore $\Vbold_\lambda = \bigcup_{n \in \omega} \Vbold_{j^n(\kappa)} \in \Vbold$. Note also that $\lambda \times \lambda \in \Vbold$ and, by the Well-Ordering Principle, for each $x \in \Vbold$ there is a bijection $$f : | \trcl(\{x\}) | \rightarrow  \trcl(\{x\}).$$ \indent Moreover, since there is an injection of $x$ into $\lambda$, we must have that $| \trcl(\{x\}) | \leq \lambda$ for each $x \in \Vbold$. Now let $$C_{x, f} \coloneqq \{ \langle \alpha, \beta \rangle \in \lambda \times \lambda : f(\alpha) \in f(\beta) \}.$$ \indent Then $C_{x, f} \in \Vbold$ and therefore so is its Mostowski collapse, with $\coll(C_{x, f}) = \trcl(\{x\})$. This means that for any $x$ and bijection $f : | \trcl(\{x\}) | \rightarrow  \trcl(\{x\})$, 
\begin{eqnarray*}
j \big( \trcl(\{x\}) \big) & = & j \big( \coll(C_{x, f}) \big) = \coll \big( j(C_{x, f}) \big) \\
 & = & \coll \big( \bigcup_{\alpha < \lambda} j(C_{x, f} \cap \Vbold_\alpha) \big)\\
 & = & \coll \big( \bigcup_{\alpha < \lambda} j \upharpoonright \Vbold_\lambda (C_{x, f} \cap \Vbold_\alpha) \big). \end{eqnarray*}  


That is, $j$ is completely determined by its construction up to $\Vbold_\lambda$. Now, let $i \coloneqq j \restrict \Vbold_\lambda$ and note that, since $\Vbold_\lambda \times \Vbold_\lambda \in \Vbold$, so is $$i = \{ \langle x, y \rangle \in \Vbold_\lambda \times \Vbold_\lambda : j(x) = y \}.$$ Therefore, by defining $\varphi(\cdot, \cdot, i, \lambda)$ as
 \begin{eqnarray*}
\varphi(x, y, i, \lambda) & \equiv & \exists f, z, C_{x, f} ~ \big( \textit{``}\dom(f)\textit{ is a cardinal''} ~ \wedge ~ \ran(f) = \trcl(\{x\}) \\
&& \wedge \textit{``f is a bijection''} ~ \wedge C_{x, f} \coloneqq \{\langle \alpha, \beta \rangle \in \lambda \times \lambda : f(\alpha) \in f(\beta)\} ~ \\
&& \wedge ~ z = \coll \big( \bigcup_{\alpha < \lambda} i(C_{x, f} \cap \Vbold_\alpha) \big) \\ && \wedge ~ \textit{``y is the element of z of maximal rank''} \big) \end{eqnarray*}

\noindent we have that $\varphi(x, y, i, \lambda)$ holds if and only if $j(x) = y$ so $j$ is definable from the parameters $i$ and $\lambda$, both of which lie in $\Vbold$, contradicting Theorem \ref{no definable}. 
\end{proof}


%% file: bigwithdc.tex


\section[Removing the assumption that V \critj is in V]{Removing the assumption that $\Vbold_{\critj} \in \Vbold$}

Assuming that \Vbold satisfies the additional assumption of \emph{dependent choice of length} $\mu$ for every infinite cardinal $\mu$, we are able to remove the assumption that $\Vbold_{\critj} \in \Vbold$. This will be done by first proving that, in this theory, every proper class must surject onto any given non-zero ordinal. In particular, for $\Vbold_{\critj}$ to be a proper class it is necessary for $\Vbold_{\critj}$ to surject onto $j(\kappa)$ which we shall show cannot happen. Note that, in the standard \ZFCbold case, the cardinality of $\Vbold_{\critj}$ is $\critj$. 

In the \ZFbold context the principle of dependent choice of length $\mu$, for $\mu$ an infinite cardinal is the following statement formulated by L\'{e}vy \cite{lev}.

\begin{quote}
Let $S$ be a non-empty set and $R$ a binary relation such that for every $\alpha \in \mu$ and every $\alpha$-sequence $s = \langle x_\beta : \beta \in \alpha \rangle$ of elements of $S$ there exists some $y \in S$ such that $s R y$. Then there is a function $f : \mu \rightarrow S$ such that for every $\alpha \in \mu$, $(f \restrict \alpha) R f(\alpha)$. 
\end{quote}

\noindent In the more general $\ZFCbold^{\uminus}$ context we want to consider a natural class version of this where $S$ and $R$ are replaced by definable classes. Such classes can be considered as the collection of sets $x$ which satisfy $\psi(x, u)$ for some formula $\psi$. This leads to the definition of the $\DCbold_\mu$\emph{-Scheme} as the following: 

\begin{quote}
Let $\varphi$ and $\psi$ be formulae and $u$ and $w$ be sets such that for some $y$, $\psi(y, u)$ and for every $\alpha \in \mu$ and every $\alpha$-sequence $s = \langle x_\beta : \beta \in \alpha \rangle$ satisfying $\psi(x_\beta, u)$ for each $\beta$, there is a $z$ satisfying $\psi(z, u)$ and $\varphi(s, z, w)$. Then there is a function $f$ with domain $\mu$ such that for each $\alpha \in \mu$ $\psi(f(\alpha), u)$ and $\varphi((f \restrict \alpha), f(\alpha), w)$.
\end{quote}

\noindent For $\mu = \aleph_0$ we shall refer to this concept as the $\DCbold$-Scheme. An equivalent way to view the $\DCbold_\mu$-Scheme is the assertion that if $T$ is a tree that has no maximal element and is $\mu$-closed, that is to say every $\alpha$-sequence of nodes in $T$ has an upper bound, then $T$ has a branch of order type $\mu$. We note that, as with the set case, if $\delta < \mu$ are infinite cardinals and $\DCbold_\mu$ holds then so does $\DCbold_\delta$. 

\begin{prop} [\cite{lev}]
If $\delta < \mu$ are infinite cardinals then the $\DCbold_\mu$-Scheme implies the $\DCbold_\delta$-Scheme.
\end{prop}

\begin{proof}
Let $\varphi$ and $\psi$ be formulae and $u$ and $w$ be sets such that for some $y$, $\psi(y, u)$ and for every $\alpha \in \mu$ and every $\alpha$-sequence $s = \langle x_\beta : \beta \in \alpha \rangle$ satisfying $\psi(x_\beta, u)$ for each $\beta$, there is a $z$ satisfying $\psi(z, u)$ and $\varphi(s, z, w)$. We define a new formula $\vartheta$ extending $\varphi$ to apply to any $\alpha$-sequence, $s$, for $\alpha \in \mu$ by 
$$\vartheta(s, z, w, \delta) \equiv \big( \alpha < \delta ~ \wedge ~ \varphi(s, z, w) \big) ~ \vee ~ \big( \alpha \geq \delta ~ \wedge ~ \psi(z, u) \big).$$
 Then for any function $f$ with domain $\mu$ witnessing this instance of the $\DCbold_\mu$-Scheme, $f \restrict \delta$ witnesses that $\DCbold_\delta$ holds for $\psi$.
\end{proof}

\noindent An important strengthening of the Collection Scheme is the Reflection Principle which we define next.

\begin{defn}
The \emph{Reflection Principle} is the assertion that for any formula $\varphi$ and set $a$ there is a transitive set $A$ such that $a \subseteq A$ and $\varphi$ is absolute between $A$ and the universe.
\end{defn}

\noindent The next pair of theorems show how this principle relates to dependent choice.

\begin{theorem} [\cite{ghj}]
Over $\ZFCbold^{\uminus}$, the $\DCbold$-Scheme is equivalent to the Reflection Principle.
\end{theorem}

\begin{theorem} [\cite{fgk}]
The Reflection Principle is not provable in $\ZFCbold^{\uminus}$.
\end{theorem}

\begin{theorem}
Suppose that $\Vbold \models \ZFbold^{\uminus} + \DCbold_\mu$ for $\mu$ an infinite cardinal. Then for any proper class $\mathcal{C}$, which is definable over $\Vbold,$ there is a subset $b$ of $\mathcal{C}$ of cardinality $\mu$.
\end{theorem}

\begin{proof}
Let $\mathcal{C} = \{ x : \psi(x) \}$ be a proper class. We shall in fact prove the equivalent statement that for any $\nu \leq \mu$ there is a subset $b$ of $\mathcal{C}$ and a bijection between $b$ and $\nu$. Suppose for a contradiction that this were not the case and let $\gamma$ be the least ordinal for which no such subset of size $\gamma$ exists. It is obvious that $\gamma$ must be an infinite cardinal. Let $\varphi(s, y)$ be the statement that $s \cup \{y\}$ is a subset of $\mathcal{C}$ and $y \not\in s$. Then, by assumption, for every $\alpha \in \gamma$ there is a sequence of elements of $\mathcal{C}$ of length $\alpha$. Also, since $\mathcal{C}$ is a proper class, if $s$ is an $\alpha$ length sequence from $\mathcal{C}$ then there is some $y \in \mathcal{C}$ which is not in $s$ so the hypothesis of $\DCbold_\gamma$ is satisfied. Therefore, by $\DCbold_\gamma$, there is a function $f$ with domain $\gamma$ and whose range gives a subset of $\mathcal{C}$ of cardinality $\gamma$, giving us our desired contradiction.
\end{proof}

\begin{cor} \label{classes are big}
Suppose that $\Vbold \models \ZFCDCminusbold$. Then for any proper class $\mathcal{C}$ which is definable over \Vbold and any non-zero ordinal $\gamma$ there is a definable surjection of $\mathcal{C}$ onto $\gamma$.
\end{cor}

\noindent We can now prove that if $j$ is a non-trivial elementary embedding from \Vbold to some class $\Mbold \subseteq \Vbold$ then $\Vbold_{\critj} \in \Vbold$.

\begin{lemma} Suppose that $\Vbold \models \ZFCDCminusbold$, $\Mbold \subseteq \Vbold$ and $j : \Vbold \rightarrow \Mbold$ is a non-trivial elementary embedding with critical point $\kappa$. Then for any $\alpha \in \kappa + 1$, $\Vbold_\alpha \in \Vbold$. 
\end{lemma}

\begin{proof}
This is proven by induction on $\alpha \in \kappa + 1$. Clearly limit cases follow by an instance of collection so it suffices to prove that for $\alpha \in \kappa$, if $\Vbold_\alpha \in \Vbold$ then so is $\Vbold_{\alpha + 1} = \mathcal{P}(\Vbold_\alpha)$. First note that $j$ fixes every set of rank less than $\kappa$ so $j \restrict \Vbold_{\alpha + 1}$ is the identity. Now suppose for sake of a contradiction that $\Vbold_{\alpha + 1}$ was a proper class. Then, by Theorem \ref{classes are big}, we could fix a set $b \subseteq \Vbold_{\alpha + 1}$ and a surjection $$h : b \twoheadrightarrow \kappa.$$ So, by elementarity, there is a surjection $$j(h) : j(b) \twoheadrightarrow j(\kappa)$$ in $\Mbold.$ However, since $b \subseteq \Vbold_{\alpha + 1}$, $j(b)$ is also a subset of $\Vbold_{j(\alpha + 1)} = \Vbold_{\alpha + 1}$ and for any $x \in j(b)$, $j(x) = x$. Therefore, $$x \in j(b) \Longleftrightarrow j(x) \in j(b) \Longleftrightarrow x \in b$$ and hence $b = j(b)$. Then, for any $x \in b$, $$j(h)(x) = j(h)(j(x)) = j(h(x)) = h(x)$$ so $j(h) = h$. But this then contradicts the assumption that $j(h)$ was a surjection onto $j(\kappa).$ Hence $\Vbold_{\alpha + 1}$ must be a set in \Vbold as required.   
\end{proof}

\vspace{12pt}

\noindent Combining this result with \Cref{no cofinal} gives the Kunen inconsistency for the theory $(\ZFCDCminusbold)_j.$

\nonewithref

\noindent However this leaves open the question as to whether or not this result is provable without relying on the Reflection Principle, namely;

\begin{question}
Is the existence of a non-trivial, cofinal, $\Sigma_0$-elementary embedding $j : \Vbold \rightarrow \Vbold$ such that $\Vbold \models \ZFCbold_j ^{\uminus}$ inconsistent?
\end{question}

\noindent The stumbling block that one needs to overcome appears to be the following:

\begin{question}
Suppose that $\Vbold \models \ZFCbold ^{\uminus}$, $\Mbold \subseteq \Vbold$ and $j : \Vbold \rightarrow \Mbold$ is a non-trivial elementary embedding. Is $\mathcal{P}(\omega) \in \Vbold$? Is $\Vbold_{\critj} \in \Vbold$?
\end{question}


%% file: collectioninsymmod3.tex


\section{Collection in Symmetric Models}\label{collection in ZF}

The important tool from the previous section was the fact that in models of  \ZFCDCminusbold proper classes are ``\emph{big}''. That is, given any non-zero ordinal, any proper class surjects onto that ordinal. This property is also a feature of models of \ZFbold as shown by the following result:

\begin{prop} Under \ZFbold, there is a surjection from any proper class onto any non-zero ordinal.
\end{prop}

\begin{proof}
Given a proper class $\mathcal{C}$, define $$S \coloneqq \{ \gamma \in \Ordbold : \exists x \in \mathcal{C} ~ \rank(x) = \gamma \}.$$ Then $S$ must be unbounded in the ordinals so, given an ordinal $\alpha$, we can take the first $\alpha$ many elements of $S$, $\{ \gamma_\beta : \beta \in \alpha \}$. Then
\begin{equation*}
f(x) = \begin{cases}
\beta, & \textit{if} ~ \rank(x) = \gamma_\beta \\
0, &\textit{otherwise}
\end{cases} 
\end{equation*} defines a surjection of $\mathcal{C}$ onto $\alpha$. \end{proof}

\noindent On the other hand, we have the following theorem from \cite{ghj}:

\begin{theorem}[\cite{ghj}]
Suppose that $\Vbold \models \ZFCbold$, $\kappa$ is a regular cardinal with $2^\omega < \aleph_\kappa$ and that $G \subseteq \Add(\omega, \aleph_\kappa)$ is $\Vbold\textrm{-generic.}$ If $\Wbold = \bigcup_{\gamma < \kappa} \Vbold[G_\gamma]$ where $G_\gamma = G \cap \Add(\omega, \aleph_\gamma)$, $($that is $G_\gamma$ is the first $\aleph_\gamma$ many of the Cohen reals added by $G)$ then $\Wbold \models \ZFCboldminus$ has the same cardinals as \Vbold and the $\DCbold_\alpha$-Scheme holds in \Wbold for all $\alpha < \kappa$, but the $\DCbold_\kappa$-Scheme and the Reflection Principle fail.  
\end{theorem}

\noindent Since \Vbold will have the same cardinals as $\Vbold[G]$, that is the full extension by all $\aleph_\kappa$ many reals, in $\Vbold[G]$ $2^\omega = \aleph_\kappa$ and therefore there is no surjection of $\mathcal{P}(\omega)$ onto $\aleph_{\kappa + 1}$. Hence there is no such surjection in $\Wbold,$ so \Wbold is a model of \ZFCboldminus ~ and the $\DCbold_\alpha$-scheme for all $\alpha \in \kappa$ in which $\mathcal{P}(\omega)$ is a proper class which is not big.

So it seems natural to ask which sub-theories of \ZFCbold also prove this feature. While we shall not answer this question here, the following results show the difficulty in coming up with a counterexample in the theory $\ZFbold^{\uminus}$.
One candidate for a counterexample in this theory is an amorphous class where:

\begin{defn}
A class $A$ is said to be \emph{amorphous} if it is infinite but not the disjoint union of two infinite subclasses. Namely, if $A = B \cup C$ then either $B$ or $C$ must be a finite set.
\end{defn}

\noindent Now one could imagine having a set $\Mbold \models \ZFbold^{\uminus}$ with $A \subseteq \Mbold$ such that \Mbold believes that $A$ is an amorphous class. Then $A$ could not surject onto $\omega$ since if $f$ were such a surjection, \mbox{$\{ x \in A : f(x) \textit{ is even} \}$} and $\{ x \in A : f(x) \textit{ is odd} \}$ would be a partition of $A$ into two infinite classes. However, we shall show that such a situation can never arise because it is inconsistent with $\ZFbold^{\uminus}$ to have a definable amorphous class. In fact, the analysis will show that a symmetric submodel of a pretame class forcing will not in general satisfy the Collection Scheme. To begin with we give the basics of symmetric extensions.

\input{symext}

\subsection{Class Forcing}\label{class forcing}

\noindent In order to describe our desired models, we shall need to do a relatively simple class forcing. Because class forcing is not the main goal of this paper, we shall only briefly sketch the construction and refer the reader to works such as \cite{fri}, \cite{hkl} or \cite{hks18} for more details on how to formally define class forcing.

As in set forcing, when trying to formalise the theory of class forcing one often works in a countable, transitive model of some second order theory such as $\GBbold^{\uminus}$. Such a model will be of the form $\mathbb{M} = \langle \Mbold, \mathcal{C} \rangle$ where \Mbold denotes the \emph{sets} of the model and $\mathcal{C}$ the classes. However, primarily for ease of notation, we shall repeatedly only talk about the first order part of the theory, noting that if \Mbold is a set model of $\ZFbold^{\uminus}$ and $\mathcal{C}$ is the collection of classes definable over \Mbold then $\langle \Mbold, \mathcal{C} \rangle$ is a model of $\GBbold^{\uminus}$. We shall say that a class $\dot{\Gamma}$ is a $\mathbb{P}$\emph{-name} if every element of $\dot{\Gamma}$ is of the form $\langle \dot{x}, p \rangle$ where $\dot{x}$ is a $\mathbb{P}$-name and $p \in \mathbb{P}$. We then define $\Mbold^\mathbb{P}$ to be the collection of $\mathbb{P}$-names which are elements of \Mbold and define $\mathcal{C}^\mathbb{P}$ as those names which are in $\mathcal{C}$. \\

\noindent Essentially, the question is which properties of set forcing are still true when the partial order, $\mathbb{P}$, is now assumed to be a proper class. Here one immediately runs into a problem when trying to prove the forcing theorem which comprises of two parts; \emph{truth} and \emph{definability}. The \emph{definability lemma} is the assertion that the forcing relation is definable in the ground model and the \emph{truth lemma} is that anything true in the generic extension is forced to be true by an element of the generic. However, as shown in \cite{hkl}, given any countable, transitive model of $\GBbold^{\uminus}$ there is a class forcing notion which does not satisfy the forcing theorem for atomic formulae. In fact, it is shown in \cite{ghh} that, over \GBbold with a global well-order, the statement that the forcing theorem holds for any class forcing is equivalent to elementary transfinite recursion for any recursion of length $\Ordbold.$ 

Moreover, even if a class forcing satisfies the forcing theorem it is not always the case that $\GBbold^{\uminus}$ will be preserved in any generic extension. The simplest such example is $\Col(\omega, \Ordbold)$ which generically adds a function collapsing the ordinals onto $\omega$. However, there is a well known collection of class forcings which both satisfy the forcing theorem and preserve all of the axioms of $\GBbold^{\uminus}$. This property was first defined by Stanley and, while such forcings are normally characterised combinatorially, for our purposes we shall use a simpler but less enlightening definition which one can prove is equivalent. 

\begin{defn}[Stanley]
Let $\mathbb{M}$ be a model of $\GBbold^{\uminus}$. A class forcing $\mathbb{P}$ is said to be \emph{pretame} for $\mathbb{M}$ if, for any generic filter $G \subseteq \mathbb{P}$, $\mathbb{M}[G]$ satisfies $\GBbold^{\uminus}$.
\end{defn}

\noindent The class forcing we shall consider is $\Add(\omega, \Ordbold)$, the forcing to add a proper class of Cohen reals. It can be proven that this satisfies the forcing theorem because it is \emph{ordinal approachable by projections},\footnote{This definition is given in \cite{hks18} as \emph{approachability by projections} however we use the terminology of \cite{hks19} where a more general definition is given.} that is to say it can be written as a continuous, increasing union of set-sized forcings, $\Add(\omega, \Ordbold) = \bigcup_{\alpha \in \Ordbold} \Add(\omega, \alpha)$. While this property in itself does not ensure that the forcing is pretame, it allows us to use an equivalent characterisation of pretameness from \cite{hks18}: 

\begin{theorem} [\cite{hks18}]\label{pretame equivalence}
Suppose that $\mathbb{M} = \langle \Mbold, \mathcal{C} \rangle$ is a model of \GBbold and $\mathbb{P}$ is a class forcing notion for $\mathbb{M}$ which satisfies the forcing theorem. Then $\mathbb{P}$ is pretame if and only if there is no set $a \in \Mbold$, name $\dot{F} \in \mathcal{C}^\mathbb{P}$ and condition $p \in \mathbb{P}$ such that $p \Vdash ``\dot{F} : \check{a} \rightarrow \Ordbold \textit{ is cofinal }"$.
\end{theorem}

\begin{theorem} Suppose that $\mathbb{M} = \langle \Mbold, \mathcal{C} \rangle$ is a model of $\GBbold + \ACbold$ and $\mathbb{P}$ is a class forcing notion for $\mathbb{M}$ which satisfies the forcing theorem. If $\mu$ is an uncountable cardinal in $\mathbb{M}$ and $\mathbb{P}$ satisfies the $\mu$-cc then $\mathbb{P}$ is pretame.
\end{theorem}

\begin{proof}
This will be proven by a variation on a standard set forcing result which uses the $\mu$-cc to approximate functions in the extension: 

\begin{claim}
Suppose that $a \in \Mbold$, $\dot{F} \in \mathcal{C}^\mathbb{P}$, $p \in \mathbb{P}$ and $p \Vdash \dot{F} : \check{a} \rightarrow \Ordbold$. Then there exists some $f \in \Mbold$ such that for all $x \in a$, $(|f(x)| \in [\Ordbold]^{< \mu})^\mathbb{M}$.
\end{claim}

To prove the claim, for each $x \in a$, let $$f(x) =\{ \alpha \in \Ordbold : \exists q \leq p ~ (q \Vdash \dot{F}(\check{x}) = \check{\alpha} ) \}.$$ Now suppose that for some $x \in a$, $f(x)$ did not have cardinality less than $\mu$. By using collection and set sized choice, we can choose a subset $Y$ of $f(x)$ of size $\mu$. Then for each $\alpha \in Y$ we can choose $q_\alpha \leq p$ such that $q_\alpha \Vdash \dot{F}(\check{x}) = \check{\alpha}$. But then this must be an antichain of size $\mu$ contradicting the assumption of $\mu$-cc.

Using the claim, we have that if $a \in \Mbold$, $\dot{F} \in \mathcal{C}^\mathbb{P}$ and $p \in \mathbb{P}$ are such that $p \Vdash \dot{F} : \check{a} \rightarrow \Ordbold$ then, taking $f$ from the conclusion of the claim, $\delta = \supremum \{ \supremum(f(x)) : x \in a \}$ is an ordinal. Therefore $$p \Vdash ``\textit{the image of } \dot{F} \textit{ is contained in } \delta"$$ so, in particular, $p$ cannot force the function to be cofinal which implies that $\mathbb{P}$ is pretame by Theorem \ref{pretame equivalence}. \end{proof}

\subsection{Breaking Collection}

\noindent We can now produce our class symmetric systems. Solely for simplicity, let $\mathbb{M}$ be a countable, transitive model of $\GBbold + \Vbold = \Lbold$. The first model we give is a symmetric system to add an amorphous class consisting of sets of Cohen reals. A detailed account of the set version of this forcing can be found in \cite{dra} and the class version of this is a simple variation. 

Let $\mathbb{P} = \Add(\omega, \Ordbold \times \omega)$ be the poset which adds $\Ordbold$ many $\omega$-blocks of Cohen reals. Now, for $\pi$ a permutation of $\Ordbold$ and $\{ \pi_\alpha : \alpha \in \Ordbold \}$ a collection of permutations of $\omega$ let $\pi$ be the permutation $$\pi : \Ordbold \times \omega \rightarrow \Ordbold \times \omega ~~~~ \pi(\alpha, n) = (\pi^0 (\alpha), \pi_\alpha(n))$$ and let $\mathcal{G}$ be the class of permutations defined in this way.  This is known as the \emph{wreath product} of the permutations of \Ordbold and the permutations of $\omega$. Then let $\mathcal{F}$ be the filter generated by $\fix(E)$ for finite sets $E \subseteq \Ordbold \times \omega.$ We can extend $\pi$ to $\mathbb{P}$-names by $$\pi p (\pi(\alpha, n)) = p(\alpha, n).$$ So the idea is that elements of $\mathcal{G}$ first permute the $\omega$-blocks of reals and then permute within the blocks. \\

\noindent Now define \begin{itemize}
\item $\dot{t}_{(\alpha, n)} \coloneqq \{ \langle \check{m}, p \rangle : p \in \mathbb{P} ~ \wedge ~ p(\alpha, n, m) = 1 \}$. This is the canonical name for the Cohen real generated by $\mathbb{P}$ restricted to the co-ordinate $(\alpha, n)$. 
\item $\dot{T}_\alpha \coloneqq \{\dot{t}_{(\alpha, n)} : n \in \omega \}^\bullet$ to be a name for the $\alpha^{th}$ $\omega$-block of reals.
\item $\dot{a} \coloneqq \{\dot{T}_\alpha : \alpha \in \Ordbold \}^\bullet$ to be a name for the collection of all \Ordbold many $\omega$-blocks. 
\end{itemize} 

\noindent One can then prove by a standard argument that in the symmetric extension $\HSbold_\mathcal{F}^G,$ $A = \dot{a}^G$ is an amorphous proper class.

The second model is the system with $\mathbb{P} = \Add(\omega, \Ordbold)$, permutations given by $\pi p (\pi \alpha, n) = p (\alpha, n)$ where $\pi$ is a permutation of \Ordbold and $\mathcal{F}$ is the filter generated by $\fix(E)$ for finite sets $E \subseteq \Ordbold$. This system will add a Dedekind-finite class of Cohen reals, $A$, and it can be proven that if $B$ is an infinite subclass of $A$ then $B$ is a proper class.

The fact that both of these symmetric submodels fail to satisfy the Collection Scheme follows from the following slightly more general theorem.

\collectionfails

\begin{proof}
To prove that the Collection Scheme fails consider classes $b$ satisfying $$\forall n \in \omega ~ \exists y \in b ~ (|y| = n ~ \wedge ~ y \subseteq A).$$ Since $\bigcup b \cap A$ is an infinite subclass of $A$, by the third assumption $b$ must be a proper class. Therefore, while for every $n \in \omega$ there is a $y$ such that $(|y| = n ~ \wedge ~ y \subseteq A)$ there is no set witnessing this for all $n$.

For the ``moreover'' part of the theorem, we shall in fact prove that any well-orderable sequence of sets can only contain finitely many elements of $A$. To see this, let $\mathcal{C} = \langle C_\alpha : \alpha \in \Ibold \rangle$ be an indexed sequence of sets where \Ibold is either \Ordbold or an infinite ordinal. We shall show that $\bigcup \mathcal{C} \cap A$ is finite and therefore that $\mathcal{C}$ cannot be a hierarchy for the universe. Suppose for a contradiction that $\bigcup \mathcal{C} \cap A$ was in fact infinite. First note that for any $\alpha \in \Ibold$, $C_\alpha \cap A$ must be finite. Now we define a sequence of ordinals $\delta_n \in \Ibold$ inductively as the least ordinal $\alpha \in \Ibold$ such that $(C_\alpha \setminus \bigcup_{m \in n} C_{\delta_m}) \cap A \neq \emptyset$. Such an ordinal must exist by the assumption that $\bigcup \mathcal{C} \cap A$ is infinite and that $\bigcup_{m \in n} (C_{\delta_m} \cap A)$ is a union of finite sets. But then $\bigcup_{n \in \omega} C_{\delta_n} \cap A$ is an infinite set, contradicting the third condition of the theorem. \end{proof}

\noindent Karagila \cite{karcom} claims that, over $\GBbold,$ any symmetric submodel of a tame\footnote{A class forcing is said to be \emph{tame} if it is pretame and also preserves Power Set.} class forcing will again be a model of $\GBbold.$ However, since it is not necessarily true that the symmetric submodel of a pretame class forcing of a model of $\GBbold^{\uminus}$ is again a model of $\GBbold^{\uminus}$ we have the following three natural general questions:

\begin{question}\label{symmetric questions} ~
\begin{itemize}
\item Suppose that $\mathbb{M} \models \GBbold^{\uminus}$. Let $\mathbb{P}$ be a pretame class forcing and $\langle \mathbb{P}, \mathcal{G}, \mathcal{F} \rangle$ a symmetric system with symmetric submodel $\Nbold.$ What theory does \Nbold satisfy?
\item Is there a combinatorial condition one can place on a symmetric system so that the symmetric submodel will satisfy $\GBbold^{\uminus}$?
\item Which class versions of sets whose existence is incompatible with choice can exist over models of \GBbold or $\GBbold^{\uminus}$?
\end{itemize}
\end{question}

\noindent In general, this final problem seems to be a difficult one to answer because, as shown by Monro in \cite{mon}, it is consistent to have a Dedekind-finite proper class:

\begin{theorem}[Monro]
Let \ZFKbold be the theory with the language of \ZFbold plus a one-place predicate \Kbold and let \Mbold be a countable transitive model of $\ZFbold.$ Then there is a model \Nbold such that \Nbold is a transitive model of \ZFKbold and \begin{eqnarray*}
\Nbold & \models & \Kbold \textit{ is a proper class which is Dedekind-finite} \\ && \textit{and can be mapped onto the universe.} 
\end{eqnarray*}
\end{theorem}

\noindent In fact it is possible to have a class symmetric system which adds a Dedekind-finite class out of \Ordbold many Cohen reals, is pretame and such that the symmetric submodel is a model of $\GBbold^{\uminus}.$ This is obtained as a symmetric submodel of the forcing $\Add(\omega, \Ordbold \times \omega)$ which adds a block of $\omega$ many Cohen reals for each ordinal. The permutation class is then those permutations which preserve the blocks and then permute within the block and the filter is, as usual, generated by the permutations which fix finite subsets of $\Ordbold \times \omega$. 

The reason why the symmetric submodel in this case models both collection and separation is that for any pair $\langle \dot{y}, p \rangle \in \Mbold^\mathbb{P} \times \mathbb{P}$ there is some ordinal $\alpha$ such that for any permutation $\pi$, $\langle \pi \dot{y}, \pi p \rangle$ can be seen as an element of the set forcing $\Add(\omega, \alpha \times \omega)$. This means that for any $H$ in the filter of subgroups, $\{ \langle \pi \dot{y}, \pi p \rangle : \pi \in H \}$ is a set which is a sufficient condition to prove the Collection Scheme. \\

\noindent To conclude this work we outline an approach to solving the first question from \ref{symmetric questions}, which we will study further in future work. As before, suppose that $\mathbb{M}$ is a countable transitive model of $\GBbold^{\uminus}$ and that $\langle \mathbb{P}, \mathcal{G}, \mathcal{F} \rangle$ is a symmetric system where $\mathbb{P}$ is a pretame class forcing. We shall say that a class name $\dot{\Gamma}$ is $\mathcal{F}$-\emph{respected} if $\{ \pi \in \mathcal{G} : \mathbbm{1} \Vdash \pi \dot{\Gamma} = \dot{\Gamma} \} \in \mathcal{F}$ and \emph{hereditarily} $\mathcal{F}$-\emph{respected} if this property holds hereditarily. Let \HRbold be the class of names $\dot{x} \in \Mbold^\mathbb{P}$ which are hereditarily respected. 

The purpose of defining symmetric names in this way is to circumvent the issue that, for an arbitrary element $H$ of the filter, $\{ \pi \dot{x} : \pi \in H \}$ may well be a proper class. However, in the set forcing case, the two resulting extensions are equal. That is to say, if $\mathbb{P}$ is a set forcing and $G \subseteq \mathbb{P}$ is generic then $\HSbold^G = \HRbold^G$ because any respected name is easily seen to be equivalent to a symmetric name by closing under its symmetry group. 

Now, suppose that our symmetric system is \emph{tenacious}\footnote{A symmetric system is tenacious if there is a dense class of conditions $p \in \mathbb{P}$ for which $\{ \pi \in \mathcal{G} : \pi p = p \} \in \mathcal{F}$. Any of the class symmetric systems we have defined are tenacious.}. Let $\dot{f}$ be a name for a function, $\dot{a}$ a name for its domain and $\dot{b}$ the obvious class forcing name for $f``a$, for example take the name defined in Theorem 3.1 of \cite{hks18}. Then, for an appropriately chosen $q \in G$ and any $\pi$ fixing $q$, $\dot{a}$ and $\dot{f}$ we will have that $\mathbbm{1} \Vdash \pi \dot{b} = \dot{b}$. Thus replacement will hold in $\HRbold^G$. Moreover, one can show that $\HRbold^G$ is a model of \GBboldminus. However, either of our earlier class symmetric systems will still witness the failure of the Collection Scheme in this model.


%% file: symext.tex


\subsection{Symmetric Extensions}

When talking about symmetric extensions we shall follow the notation of Karagila, which can be found in papers such as \cite{kar} and \cite{kar20}. The purpose of symmetric extensions is to find a suitable submodel of a generic extension of \ZFCbold in which choice fails. This allows us build objects such as Dedekind-finite sets and amorphous sets whose existence would contradict the axiom of choice while still using choice in the ground model to control what happens.

\begin{nota}
Given a set of $\mathbb{P}$-names $\{ \dot{x}_i : i \in I \}$ let $\{ \dot{x}_i : i \in I \}^\bullet$ denote the $\mathbb{P}$-name \mbox{$\{ \langle \dot{x}_i, \mathbbm{1} \rangle : i \in I \}$.}
\end{nota}

\noindent Now, given a forcing notion $\mathbb{P}$ and automorphism $\pi$ of $\mathbb{P}$, $\pi$ can be extended to $\mathbb{P}$-names by the following recursion: $$ \pi \dot{x} = \{ \langle \pi \dot{y}, \pi p \rangle : \langle \dot{y}, p \rangle \in \dot{x} \}.$$ For $\mathcal{G}$ a group of automorphisms of $\mathbb{P}$ we say that $\mathcal{F}$ is a normal filter of subgroups over $\mathcal{G}$ if: 

\begin{itemize}
\item $\mathcal{F}$ is a non-empty family of subgroups of $\mathcal{G}$.
\item $\mathcal{F}$ is closed under finite intersections and supergroups.
\item (\emph{Normality}) For any $H \in \mathcal{F}$ and $\pi \in \mathcal{G}$, $\pi H \pi^{-1} \in \mathcal{F}$.
\end{itemize} 
A symmetric system is then a triple $\langle \mathbb{P}, \mathcal{G}, \mathcal{F} \rangle$ such that $\mathbb{P}$ is a notion of forcing, $\mathcal{G}$ is a group of automorphisms of $\mathbb{P}$ and $\mathcal{F}$ is a normal filter of subgroups over $\mathcal{G}$. 

We also fix the following notation: 

\begin{itemize}
\item The \emph{stabaliser} of $\dot{x}$ under $\mathcal{G}$ is $\sym_\mathcal{G} (\dot{x}) = \{ \pi \in \mathcal{G} : \pi \dot{x} = \dot{x} \}$.
\item $\dot{x}$ is said to be $\mathcal{F}$\emph{-symmetric} when $\sym_\mathcal{G} (\dot{x}) \in \mathcal{F}$.
\item $\dot{x}$ is said to be \emph{hereditarily}  $\mathcal{F}$\emph{-symmetric} when it is $\mathcal{F}$-symmetric and for any $\langle p, \dot{y} \rangle$ in $\dot{x}$, $\dot{y}$ is also hereditarily  $\mathcal{F}$-symmetric. 
\item $\HSbold_\mathcal{F}$ denotes the class of hereditarily $\mathcal{F}$-symmetric names.
\end{itemize}

\noindent We then have the following two theorems which summarise the important consequences of this construction and can be found in chapter 15 of \cite{jec}.

\begin{lemma}[The Symmetry Lemma] For any $p \in \mathbb{P}$, automorphism $\pi$ of $\mathbb{P}$, formula $\varphi(v)$ of the forcing language and $\mathbb{P}$-name $\dot{x}$, $$p \Vdash \varphi(\dot{x}) \Longleftrightarrow \pi p \Vdash \varphi(\pi \dot{x}).$$
\end{lemma}

\begin{theorem}
If \Mbold is a transitive model of $\ZFbold,$ $\langle \mathbb{P}, \mathcal{G}, \mathcal{F} \rangle$ is a symmetric system and $G \subseteq \mathbb{P}$ is a generic filter then $\Nbold \coloneqq \HSbold^G_\mathcal{F} = \{ \dot{x}^G : \dot{x} \in \HSbold_\mathcal{F} \}$ is a transitive model of \ZFbold with $\Mbold \subseteq \Nbold \subseteq \Mbold[G]$.
\end{theorem}
